\documentclass[12pt]{article}

\usepackage{authblk}
\usepackage{hyperref}
\usepackage[T1]{fontenc}
\usepackage{amsmath}
\usepackage{amsfonts,amssymb}
\usepackage{geometry}
\usepackage{xcolor}
\usepackage{enumerate}
\usepackage{enumitem}
\usepackage{verbatim}
\usepackage[utf8]{inputenc}
\usepackage{mathtools}
\usepackage{graphicx}
\usepackage{multicol}
\usepackage{multirow}
\usepackage{pifont}
\usepackage[numbers,sort&compress]{natbib}
\usepackage{bm}
\hypersetup{
  hidelinks,
  pdftitle={Inversion Diameter of Planar Graphs},
  pdfauthor={Yichen Wang and Yuxuan Yang},
  pdfkeywords={inversion diameter, orientation, planar graph, acyclic coloring, strong degeneracy}
}

\def\0{\emptyset}

\newtheorem{theorem}{Theorem}[section]
\newtheorem{definition}[theorem]{Definition}
\newtheorem{lemma}[theorem]{Lemma}

\newenvironment{proof}[1][Proof]{\noindent\textit{#1.}\ }{\hfill $\square$\par}

\DeclareMathOperator{\tw}{\mathrm{tw}}
\DeclareMathOperator{\diam}{\operatorname{diam}}

\begin{document}


\title{Inversion Diameter of Planar Graphs}
\author[1]{\small\bfseries Yichen Wang\thanks{E-mail: wangyich22@mails.tsinghua.edu.cn}}
\author[2]{\small\bfseries Yuxuan Yang\thanks{Corresponding author. E-mail: yangyx@bupt.edu.cn}}
\affil[1]{\small Department of Mathematical Sciences, Tsinghua University, Beijing 100084, China}
\affil[2]{\small School of Science, Beijing University of Posts and Telecommunications, Beijing 100876, China}

\date{}

\maketitle

\begin{abstract}
Given an oriented graph $\vec{G}$ and a subset of vertices $X \subseteq V(\vec{G})$, the \emph{inversion} of $X$ is the operation that reverses the orientation of every arc with both endpoints in $X$. For a simple graph $G$, the inversion diameter $\diam(I(G))$ is the maximum distance between two orientations of $G$ under inversions of vertex sets. We prove the
sharp bound
\[
\diam(I(G))\le 2\chi_a(G)-2,
\]
where $\chi_a(G)$ is the acyclic chromatic number. Consequently, every planar
graph has inversion diameter at most $8$, improving the previously known
bound $12$. Using strong-degeneracy arguments, we also obtain upper bounds
$7$, $5$, and $4$ for planar graphs of girth at least $4$, $5$, and $6$,
respectively.
\end{abstract}


\vskip.3cm

\section{Introduction}

Let $\vec G$ be an orientation of a graph $G$. For a set
$X\subseteq V(G)$, the \emph{inversion} of $X$ reverses every arc of
$\vec G[X]$. This operation was introduced by Belkhechine et
al.~\cite{BBBP10} in connection with the inversion number, which asks for the
minimum number of inversions that makes a given orientation acyclic. The
operation has since been studied in several structural and algorithmic
settings; see, for example, \cite{APS+24,AHH+22,BdSH21,DHHR23}.

The \emph{inversion graph} $I(G)$ has as its vertices the labeled
orientations of $G$, with two orientations adjacent when one is obtained from
the other by a single inversion. Its diameter $\diam(I(G))$ is therefore the
largest inversion distance between two orientations of $G$. Equivalently,
after recording the edges on which two orientations disagree by a labeling
$\pi:E(G)\to\mathbb F_2$, it is the least integer $t$ such that every such
labeling has a representation
\[
\pi(uv)=\bm f(u)\cdot\bm f(v)\qquad(uv\in E(G))
\]
by vectors in $\mathbb F_2^t$. This linear-algebraic formulation is used throughout the paper.

Havet et al.~\cite{HHR24} related bounded inversion diameter to bounded star,
acyclic, and oriented chromatic numbers. They also proved that every planar
graph has inversion diameter at most $12$, that every planar graph of girth
at least $8$ has inversion diameter at most $3$, and that the latter value is
best possible even for planar graphs of arbitrarily large even girth. Arana
et al.~\cite{ABBC+} subsequently obtained the lower bound $6$ for unrestricted
planar graphs, improved several bounds under girth assumptions, and determined
the value $3$ for girth at least $7$. Before the present work, the best bounds
for unrestricted planar graphs were therefore $6$ and $12$.

Our first result gives a sharp linear relation with the acyclic chromatic
number. Havet et al.~\cite{HHR24} proved the quadratic estimate
$\diam(I(G))\le(2k^2+2k+2)/3$ when $\chi_a(G)=k$, and Arana et
al.~\cite{ABBC+} improved this to $\max\{4k-7,2\}$. We prove the following.

\begin{theorem}\label{thm:main}
If a graph $G$ has acyclic chromatic number $k$, then
\begin{equation}\label{eq:main-bound}
\diam(I(G)) \le 2k-2.
\end{equation}
\end{theorem}

The bound is exact for every $k\ge2$: there is a graph $G$ with
$\chi_a(G)=k$ and
$\diam(I(G))=2k-2$. This follows by combining Theorem~\ref{thm:main} with
the sharp treewidth examples constructed in~\cite{WWYL25}. Borodin's acyclic
$5$-color theorem for planar graphs now gives our first planar consequence.

\begin{theorem}\label{thm:planar8}
Every planar graph $G$ satisfies $\diam(I(G)) \le 8$.
\end{theorem}

Stronger estimates are possible when the girth is prescribed. For these
results we use directly the strong degeneracy introduced by Havet et
al.~\cite{HHR24}: a strongly $t$-degenerate graph has inversion diameter at
most $t$. We encode a recursive strong ordering by an auxiliary plane graph
whose remaining original vertices are called variable vertices and whose
removed vertices are represented by constrain vertices. The degree of a
constrain vertex measures the number of subsets that must be avoided in
Havet's ordering condition.

For every $g\ge5$, Euler's formula yields a uniform charge computation and an
explicit strong-degeneracy bound $T(g)$. It gives
$T(5)=5$, $T(6)=T(7)=4$, and $T(g)=3$ for $g\ge8$. The case $g=4$ is a
genuine boundary case: the Euler coefficient of a binary constrain vertex is
zero although that vertex still contributes a forbidden subset. We overcome
this obstruction by suppressing the binary constrain vertices, combining the
resulting planar estimate with the triangle-free Euler estimate, and deriving
a separate reducible-vertex lemma. This proves the following three new
bounds.

\begin{theorem}\label{thm:girth4-intro}
Every finite simple planar graph $G$ of girth at least $4$ satisfies
$\diam(I(G))\le 7$.
\end{theorem}

\begin{theorem}\label{thm:girth5-intro}
Every finite simple planar graph $G$ of girth at least $5$ satisfies
$\diam(I(G))\le 5$.
\end{theorem}

\begin{theorem}\label{thm:girth6-intro}
Every finite simple planar graph $G$ of girth at least $6$ satisfies
$\diam(I(G))\le 4$.
\end{theorem}

These results are summarized in Table~\ref{tab:girth-bounds}.
\begin{table}[htbp]
\centering
\begin{tabular}{|c|c|c|c|}
\hline
$g$ & lower bound & previous upper bound & upper bound proved here\\
\hline
3 & 6~\cite{ABBC+} & 12~\cite{HHR24} & 8 \\
\hline
4 & 4~\cite{ABBC+} & 10~\cite{ABBC+} & 7 \\
\hline
5 & 3~\cite{HHR24} & 7~\cite{ABBC+} & 5 \\
\hline
6 & 3~\cite{HHR24} & 5~\cite{ABBC+} & 4\\
\hline
7 & 3~\cite{HHR24} & 3~\cite{ABBC+} & -- \\
\hline
8 & 3~\cite{HHR24} & 3~\cite{HHR24} & -- \\
\hline
\end{tabular}
\caption{Bounds on the inversion diameter of planar graphs of minimum girth
$g$.}
\label{tab:girth-bounds}
\end{table}

The paper is organized as follows. Section~2 collects the vector
characterization, the required coloring results, and strong degeneracy.
Section~3 proves Theorem~\ref{thm:main}, its sharpness, and
Theorem~\ref{thm:planar8}. Section~4 develops the plane constraint systems,
gives the uniform computation for $g\ge5$, and concludes with the exceptional
girth-$4$ argument.

\paragraph{Notation.} We use standard graph-theoretic notation; see
\cite{BM08} for undefined terms. All graphs are finite, simple, and
undirected unless stated otherwise. We write $\mathbb{F}_2$ for the $2$-element
field and $\cdot$ for the standard dot product on $\mathbb{F}_2^t$. For a positive
integer $k$, we let $[k]=\{1,\dots,k\}$.

\section{Preliminaries}

We recall the recursive definition of $k$-trees and the resulting subgraph
characterization of treewidth.

A \emph{$k$-tree} is a graph obtained from the complete graph $K_{k+1}$ by
repeatedly inserting new vertices linked to an existing clique of size $k$.
A graph is then said to have \emph{treewidth at most $k$} if it is a subgraph
of some $k$-tree. It is well known that this definition coincides with the usual
definition via tree-decompositions~\cite{Bodlaender98}.

The following folklore result appears, for instance, in the work of
Sopena~\cite{Sopena01}; we include a proof for completeness.

\begin{theorem}\label{tw_acyclic}
Every graph with treewidth at most $k$ has acyclic chromatic number at most $k+1$.
\end{theorem}

\begin{proof}
It suffices to prove the statement for $k$-trees, since the acyclic chromatic
number is monotone under taking subgraphs. Starting with a proper
$(k+1)$-coloring of the initial $(k+1)$-clique, every newly inserted vertex $v$
is attached to an existing $k$-clique $K$. By induction, the $k$ vertices of
$K$ have received $k$ distinct colors, so $v$ can be colored with the unique
remaining color from $\{1,\dots,k+1\}$. This yields a proper coloring.
To see that the coloring is acyclic, argue inductively along the construction
of the $k$-tree. If a bichromatic cycle first appears when $v$ is inserted,
then the two neighbors of $v$ on this cycle lie in $K$ and have the same
color. This is impossible because the vertices of the clique $K$ have
pairwise distinct colors.
\end{proof}

\begin{theorem}[Wang et al.~\cite{WWYL25}]\label{thm:wwyl}
For every integer $t\ge 1$, there exists a graph $G$ such that
\[
\tw(G)=t
\qquad\text{and}\qquad
\diam(I(G))=2t.
\]
\end{theorem}

We shall make repeated use of the following linear-algebraic characterization of
the inversion diameter, which is essentially Observation~2.2 in Havet et al.~\cite{HHR24}.

\begin{lemma}\label{lem:vector_rep}
For every graph $G$ and every positive integer $t$, the following are equivalent:
\begin{enumerate}
    \item $\diam(I(G)) \le t$;
    \item for every edge labeling $\pi:E(G)\to\mathbb{F}_2$, there exists a map $\bm{f}:V(G)\to\mathbb{F}_2^{t}$ such that $\bm{f}(u)\cdot \bm{f}(v)=\pi(uv)$ for every edge $uv\in E(G)$.
\end{enumerate}
\end{lemma}

The second tool, and the one used directly in all of our girth arguments, is
the strong degeneracy introduced by Havet et al.~\cite{HHR24}. Its purpose is
to identify exactly what must be controlled when the vectors in
Lemma~\ref{lem:vector_rep} are assigned in the order of the vertices. For a
vertex $u$, the already assigned neighbors prescribe linear equations. At
the same time, the earlier neighborhoods of vertices occurring after $u$
describe the subspaces that the vector of $u$ must avoid. Strong degeneracy
combines these two quantities in a single ordering condition.

We now give the definition in the form introduced by Havet et
al.~\cite{HHR24}. Let $\prec$ be a total ordering of $V(G)$. For
$u,v\in V(G)$, put
\[
N^-_\prec(u)=\{x\in N(u):x\prec u\},\qquad
N^+_\prec(u)=\{x\in N(u):u\prec x\},
\]
and
\[
N_{\prec u}(v)=\{x\in N(v):x\prec u\}.
\]
Define
\[
\mathcal S_\prec(u)=
\{Y\subseteq V(G):
  \text{there is }v\in N^+_\prec(u)\text{ with }
  Y\subseteq N_{\prec u}(v)\}.
\]

\begin{definition}[Havet et al.~\cite{HHR24}]\label{def:strong-degeneracy}
All logarithms are to base $2$. The ordering $\prec$ is \emph{$t$-strong}
if, for every $u\in V(G)$,
\[
|N^-_\prec(u)|+\log |\mathcal S_\prec(u)|<t
\quad\text{when }N^+_\prec(u)\ne\emptyset,
\]
and $|N^-_\prec(u)|\le t$ when $N^+_\prec(u)=\emptyset$. A graph is
\emph{strongly $t$-degenerate} if it has a $t$-strong ordering.
\end{definition}

The decisive advantage is that a $t$-strong ordering can be used without any
further graph-specific linear algebra. The following theorem of Havet et
al.~\cite{HHR24} converts the ordering directly into the required bound on
the inversion diameter.

\begin{theorem}[Havet et al.~\cite{HHR24}]\label{thm:strong-diameter}
If a graph $G$ is strongly $t$-degenerate, then
\[
\diam(I(G))\le t.
\]
\end{theorem}

To indicate why this theorem applies, fix an edge labeling and assign the
vectors in increasing order. The vectors already assigned to
$N^-_\prec(u)$ are kept linearly independent whenever they form the earlier
neighborhood of an unassigned vertex. The equations at $u$ leave an affine
space of size at least $2^{t-|N^-_\prec(u)|}$, while the union of the spans
that must be avoided has size at most $|\mathcal S_\prec(u)|$. The strict
inequality in Definition~\ref{def:strong-degeneracy} leaves an admissible
vector. Accordingly, in Section~\ref{sec:strong-girth} we work only on
constructing $t$-strong orderings and then apply
Theorem~\ref{thm:strong-diameter} directly.

Recall that a proper coloring $c:V(G)\to[k]$ is \emph{acyclic} if, for every
pair of distinct colors $i,j\in[k]$, the subgraph induced by the vertices of
colors $i$ and $j$ is a forest. The minimum number of colors in an acyclic
coloring of $G$ is the \emph{acyclic chromatic number}, denoted $\chi_a(G)$.

The following results on acyclic colorings of planar graphs will be used.

\begin{theorem}[Borodin~\cite{Borodin79}]\label{thm:borodin}
Every planar graph has acyclic chromatic number at most $5$.
\end{theorem}

\begin{theorem}[Borodin et al.~\cite{BKW99}]\label{thm:bkw}
Let $G$ be a planar graph.
\begin{enumerate}
    \item If $G$ has girth at least $5$, then $\chi_a(G)\le 4$.
    \item If $G$ has girth at least $7$, then $\chi_a(G)\le 3$.
\end{enumerate}
\end{theorem}

\section{Inversion Diameter and Acyclic Chromatic Number}

We divide the proof of Theorem~\ref{thm:main} into two steps. First, we use an
acyclic coloring to construct a structured inversion strategy of dimension
$2k$. We then exploit this structure to reduce the dimension to $2k-2$.

\begin{proof}[Proof of Theorem~\ref{thm:main}]
If $k=1$, then $G$ has no edges, and the conclusion is trivial. Assume
$k\ge 2$.
We shall construct, for every edge labeling $\pi:E(G)\to\mathbb F_2$, a map
\[
\bm{f}^{*}:V(G)\to\mathbb F_2^{2k-2}
\]
such that
\[
\bm{f}^{*}(u)\cdot \bm{f}^{*}(v)=\pi(uv)
\qquad\text{for every }uv\in E(G).
\]
By Lemma~\ref{lem:vector_rep} this will prove the theorem.

\textbf{Step 1.} We first construct a map
\[
\bm{f}:V(G)\to\mathbb F_2^{2k}
\]
such that
\[
\bm{f}(u)\cdot \bm{f}(v)=\pi(uv)
\qquad\text{for every }uv\in E(G).
\]
Let $\bm{e}_1,\bm{e}_2,\dots,\bm{e}_{2k}$ be the standard basis of $\mathbb F_2^{2k}$. Define
\[
\bm{q}_r=\bm{e}_{2r-1}+\bm{e}_{2r},\qquad r=1,\dots,k,
\]
and let
\[
\mathcal{Q}=\langle \bm{q}_1,\dots,\bm{q}_k\rangle.
\]
Since $1+1=0$ in $\mathbb{F}_2$, the subspace $\mathcal{Q}$ is totally isotropic:
\[
\bm{q}_r\cdot \bm{q}_s=0
\qquad\text{for all }r,s\in[k].
\]
Next, assign a distinct vector outside $\mathcal{Q}$ for each color class.
Let
\[
\bm{p}_i=\bm{e}_{2i},\qquad i=1,\dots,k.
\]
These vectors act as linear functionals on $\mathcal{Q}$: for $i=1,\dots,k$ and $r=1,\dots,k$,
\[
\lambda_i(\bm{q}_r):=\bm{p}_i\cdot \bm{q}_r=\delta_{ir},
\]
so $\{\lambda_i:i\in[k]\}$ is the standard basis of $\mathcal{Q}^*$.

Now fix an edge labeling $\pi:E(G)\to\mathbb F_2$.
Let $c:V(G)\to[k]$ be the given acyclic coloring. For each unordered pair
of distinct colors $\{i,j\}$, consider the bichromatic subgraph
\[
G_{ij}=G[c^{-1}(\{i,j\})].
\]
Since the coloring is acyclic, $G_{ij}$ is a forest. On this forest we
introduce variables as follows. If $u$ has color $i$, it receives a
variable $x_j(u)$; if $v$ has color $j$, it receives a variable
$x_i(v)$. For every edge $uv\in E(G_{ij})$ with $c(u)=i$ and
$c(v)=j$, impose the equation
\[
x_j(u)+x_i(v)=\pi(uv).
\]
This system is solvable because $G_{ij}$ is a forest: on each connected
component choose one root variable arbitrarily, and then propagate the values
along the tree edges. There is no cycle, so no consistency condition can
fail. Doing this independently for all pairs $\{i,j\}$, we obtain, for each
vertex $v$ of color $i$, values
\[
x_j(v)\in\mathbb F_2\qquad\text{for every }j\ne i.
\]
For convenience we also set $x_i(v)=0$.
Let
\[
\bm{q}(v)=\sum_{j=1}^kx_j(v)\,\bm{q}_j.
\]
Then for every $i\in[k]$,
\[
\bm{p}_i\cdot \bm{q}(v)=\sum_{j=1}^kx_j(v)\,\bm{p}_i\cdot \bm{q}_j = x_i(v).
\]
Define
\[
\bm{f}(v)=\bm{p}_{c(v)}+\bm{q}(v).
\]
We now check an edge. Let $uv\in E(G)$, and write $c(u)=i$,
$c(v)=j$ with $i\ne j$. Then
\[
\begin{aligned}
\bm{f}(u)\cdot \bm{f}(v)
&=(\bm{p}_i+\bm{q}(u))\cdot(\bm{p}_j+\bm{q}(v))\\
&=\bm{p}_i\cdot \bm{p}_j+\bm{p}_i\cdot \bm{q}(v)+\bm{p}_j\cdot \bm{q}(u)+\bm{q}(u)\cdot \bm{q}(v)\\
&=0+x_i(v)+x_j(u)+0\\
&=\pi(uv).
\end{aligned}
\]
Thus the required map $\bm{f}$ exists for every $\pi$.

\textbf{Step 2.} We now reduce the dimension to $2k-2$.
Observe that in the construction above, $\bm{q}_i$ is orthogonal to $\bm{f}(v)$ for every vertex $v$ with $c(v)\neq i$. Consequently, replacing $\bm{f}(v)$ by $\bm{f}(v)+\bm{q}_{c(v)}$ does not change the value of $\bm{f}(u)\cdot \bm{f}(w)$ for any edge $uw\in E(G)$. This indicates that the construction carries redundant degrees of freedom, which we now eliminate.

By ignoring the last two coordinates $\bm{e}_{2k-1}$ and $\bm{e}_{2k}$, we regard
\[
\bm{p}_1,\dots,\bm{p}_{k-1},\quad \bm{q}_1,\dots,\bm{q}_{k-1}
\]
as vectors in $\mathbb{F}_2^{2k-2}$. Define
\[
\bm{p}_k^*:=\sum_{i=1}^{k-1} \bm{p}_i.
\]
For a vertex $v$ with $c(v)=k$, set
\[
\bm{q}^{*}(v)=\sum_{i=1}^{k-1}x_i(v)\,\bm{q}_i,
\qquad
\bm{f}^{*}(v)=\bm{p}_k^*+\bm{q}^{*}(v).
\]
For a vertex $v$ with $c(v)=i\neq k$, set
\[
\alpha_i(v):=1+\sum_{j=1}^{k}x_j(v),
\]
\[
\bm{q}^{*}(v)=\alpha_i(v)\,\bm{q}_i+\sum_{j=1}^{k-1}x_j(v)\,\bm{q}_j,
\qquad
\bm{f}^{*}(v)=\bm{p}_i+\bm{q}^{*}(v).
\]

We verify that $\bm{f}^{*}(u)\cdot \bm{f}^{*}(v)=\pi(uv)$ for every edge $uv\in E(G)$.
Write $c(u)=i$ and $c(v)=j$ with $i\ne j$.

\textit{\textbf{Case 1}: $i,j\in[k-1]$.} Since $\bm{q}^{*}(u),\bm{q}^{*}(v)\in\langle \bm{q}_1,\dots,\bm{q}_{k-1}\rangle$ and this subspace is totally isotropic, we have $\bm{q}^{*}(u)\cdot \bm{q}^{*}(v)=0$. Moreover $\bm{p}_i\cdot \bm{p}_j=0$. Hence
\[
\begin{aligned}
\bm{f}^{*}(u)\cdot \bm{f}^{*}(v)
&=\bm{p}_i\cdot \bm{q}^{*}(v)+\bm{p}_j\cdot \bm{q}^{*}(u)\\
&=x_i(v)+x_j(u)=\pi(uv).
\end{aligned}
\]

\textit{\textbf{Case 2}: $i=k$ and $j\in[k-1]$.} Again $\bm{q}^{*}(u)\cdot \bm{q}^{*}(v)=0$. We have $\bm{p}_k^*\cdot \bm{p}_j=1$ and
\[
\bm{p}_k^*\cdot \bm{q}^{*}(v)
=\Bigl(\sum_{r=1}^{k-1}\bm{p}_r\Bigr)\cdot\Bigl(\alpha_j(v)\,\bm{q}_j+\sum_{m=1}^{k-1}x_m(v)\,\bm{q}_m\Bigr)
=\alpha_j(v)+\sum_{m=1}^{k-1}x_m(v)
=1+x_k(v).
\]
Moreover $\bm{p}_j\cdot \bm{q}^{*}(u)=x_j(u)$. Therefore,
\[
\bm{f}^{*}(u)\cdot \bm{f}^{*}(v)
=1+\bigl(1+x_k(v)\bigr)+x_j(u)
=x_k(v)+x_j(u)=\pi(uv).
\]

\textit{\textbf{Case 3}: $i\in[k-1]$ and $j=k$.} This is symmetric to Case~2.

All cases agree, so the map $\bm{f}^{*}:V(G)\to\mathbb{F}_2^{2k-2}$ has the required property. By Lemma~\ref{lem:vector_rep}, this completes the proof of Theorem~\ref{thm:main}.
\end{proof}

\begin{proof}[Proof of Theorem~\ref{thm:planar8}]
By Theorem~\ref{thm:borodin}, every planar graph admits an acyclic $5$-coloring. Applying Theorem~\ref{thm:main} with $k=5$ gives $\diam(I(G))\le 2\cdot 5-2=8$.
\end{proof}

Theorems~\ref{thm:borodin}, \ref{thm:bkw}, and \ref{thm:main} yield the
preliminary bounds $8$ for planar graphs of girth at least $4$ and $6$ for
planar graphs of girth at least $5$ (and hence also at least $6$). In
Section~\ref{sec:strong-girth}, we improve these bounds to $7$, $5$, and $4$,
respectively.

We now prove that Theorem~\ref{thm:main} is sharp.

\begin{theorem}\label{thm:tight}
For every integer $k\ge 2$, there exists a graph $G$ with acyclic chromatic
number $k$ such that
\[
\diam(I(G))=2k-2.
\]
\end{theorem}
\begin{proof}
By Theorem~\ref{thm:wwyl}, for every integer $t\ge 1$ there exists a graph of treewidth $t$ whose inversion diameter equals $2t$. Taking $t=k-1$, we obtain a graph $G$ with treewidth $k-1$ and $\diam(I(G))=2k-2$. By Theorem~\ref{tw_acyclic}, $G$ has acyclic chromatic number at most $k$, while Theorem~\ref{thm:main} implies that its acyclic chromatic number is at least $k$ (otherwise the inversion diameter would be strictly smaller than $2k-2$). Hence $\chi_a(G)=k$ and the bound is attained.
\end{proof}

\section{Strong Degeneracy and Planar Graphs}\label{sec:strong-girth}

Theorem~\ref{thm:strong-diameter} is the starting point of this section. Thus
we do not construct vector representations separately for girths $4$, $5$,
and $6$. Instead, we construct strong orderings and apply the theorem of
Havet et al.~\cite{HHR24} directly. The same ordering calculation treats
every girth at least $5$ and, in particular, proves
Theorems~\ref{thm:girth5-intro} and \ref{thm:girth6-intro}. The case of girth
$4$ is the limiting case of
the calculation and requires one additional planar estimate. We first give
the uniform computation for $g\ge5$. We then return to girth $4$ at the end
of the section, explain exactly why the uniform parameters break down, and
replace the missing estimate by a separate argument.

We first explain the auxiliary graph used to expose the strong-ordering
condition. Imagine constructing the final ordering from right to left by
successively removing vertices. The original vertices that have not yet
been removed are called variable vertices. Once a vertex is removed, it
will occur later in the final ordering; we replace it by a constrain vertex
adjacent to its neighbors that are still variable vertices. Consequently,
an incident constrain vertex of degree $r$ records an earlier neighborhood
of size $r-1$ and contributes at most $2^{r-1}-1$ nonempty subsets to
$\mathcal S_\prec(u)$. This is precisely the second term in Havet's
$t$-strong condition.

\begin{definition}\label{def:plane-constraint-system}
A \emph{plane constraint system} is a plane graph $P$ with a partition
$V(P)=X\cup A$, where the vertices in $X$ are called \emph{variable
vertices} and those in $A$ are called \emph{constrain vertices}. There is
no edge with both endpoints in $A$. An edge with both endpoints in $X$ is a
\emph{direct edge}, and every constrain vertex has degree at least $2$. For
$v\in X$, let $q_P(v)$ be its direct degree and put
\begin{equation}\label{eq:constraint-weight}
 W_P(v)=\sum_{a\in N_P(v)\cap A}\bigl(2^{d_P(a)-1}-1\bigr).
\end{equation}
The system is \emph{canonical} if distinct constrain vertices have distinct
neighborhoods in $X$.
\end{definition}

To construct the ordering recursively, we must update the auxiliary graph
after choosing its next variable vertex. The following reduction does this
while preserving the information needed for
Definition~\ref{def:strong-degeneracy}. The \emph{reduction} of $P$ at a
variable vertex $v$ is defined as follows. Delete $v$ and all its incident
edges, discard every constrain vertex whose degree becomes at most $1$, and,
when $q_P(v)\ge2$, introduce a new constrain vertex adjacent precisely to the
direct neighbors of $v$. Finally, retain only one copy of each constrain
vertex neighborhood. The new constrain vertex can be placed at the former
position of $v$, with its incident edges drawn along the deleted direct edges.
Consequently,
the reduction preserves planarity and does not decrease the girth. It also
preserves canonicality.

The next lemma makes the connection with Havet's definition explicit. At
the moment when a variable vertex $v$ is removed, $q_P(v)$ will be its number
of earlier neighbors, while $1+W_P(v)$ bounds
$|\mathcal S_\prec(v)|$. Therefore the two inequalities below are exactly
what is needed for a $t$-strong ordering.

\begin{lemma}\label{lem:constraint-order}
Let $G$ be a plane graph and let $t$ be a positive integer. Start with the
constraint system having variable-vertex set $V(G)$, direct-edge set $E(G)$,
and no constrain vertices. Suppose that, until no variable vertex remains,
one can reduce at a variable vertex $v$ satisfying
\begin{equation}\label{eq:strong-reduction}
 q_P(v)<t
 \qquad\text{and}\qquad
 W_P(v)<2^{t-q_P(v)}-1.
\end{equation}
Then $G$ is strongly $t$-degenerate.
\end{lemma}

The proof reads the reduction sequence backwards and checks the two terms in
Definition~\ref{def:strong-degeneracy} separately.

\begin{proof}
Let $v_1,\dots,v_n$ be the reduction sequence, and order the vertices in the
reverse order:
\[
v_n\prec v_{n-1}\prec\cdots\prec v_1.
\]
Consider the moment at which $v_i$ is reduced. At this moment the set of
variable vertices is
\[
\{v_i,v_{i+1},\dots,v_n\}.
\]
Hence its direct neighbors at this
moment are the earlier neighbors of $v_i$ in the final ordering, and
\begin{equation}\label{eq:q-earlier}
q_P(v_i)=|N^-_\prec(v_i)|.
\end{equation}

Every constrain vertex incident with $v_i$ records a later neighbor $v_j$, where
$j<i$, together with the set $N_{\prec v_i}(v_j)$. Constrain vertices with
the same recorded set may have been
identified, which does not enlarge a union of power sets. Later neighbors
with empty recorded set contribute only the empty set. It follows that
\[
|\mathcal S_\prec(v_i)|
 \le 1+\sum_{a\in N_P(v_i)\cap A}
       \bigl(2^{d_P(a)-1}-1\bigr)
 =1+W_P(v_i).
\]
If $v_i$ has a later neighbor, \eqref{eq:strong-reduction} therefore gives
\[
|N^-_\prec(v_i)|+\log|\mathcal S_\prec(v_i)|<t.
\]
If it has no later neighbor, then \eqref{eq:q-earlier} and
\eqref{eq:strong-reduction} give $|N^-_\prec(v_i)|<t$. Thus $\prec$ is a
$t$-strong ordering.
\end{proof}

\subsection{A uniform computation for girth at least five}

We now explain the parameters used in the uniform calculation. Fix
$g\ge5$. The Euler bound for a plane graph of girth at least $g$ has density
coefficient $g/(g-2)$; we denote this coefficient by $\alpha_g$. Our charge
argument will find a variable vertex whose direct degree $q$ satisfies
$q/2<\alpha_g$. Hence $q<2\alpha_g$, and $D_g$ below is the largest integer
that $q$ can be. Thus the first two parameters encode only the planar density
and its immediate degree consequence:
\begin{equation}\label{eq:general-parameters}
\alpha_g=\frac{g}{g-2},\qquad
D_g=\lceil2\alpha_g\rceil-1.
\end{equation}

An incident constrain vertex of degree $r$ contributes
$1-\alpha_g/r$ to the local Euler charge and contributes
$2^{r-1}-1$ to $W_P(v)$. To convert an upper bound on the first quantity
into an upper bound on the second, we take the largest possible ratio:
\begin{equation}\label{eq:rho-g}
\rho_g=
\max_{2\le r\le D_g}
\frac{2^{r-1}-1}{1-\alpha_g/r}.
\end{equation}

Finally, suppose that the selected variable vertex has direct degree $q$.
The integer $C_g(q)$ is an upper bound for $1+W_P(v)$. We choose $T(g)$ so
that $C_g(q)<2^{T(g)-q}$ simultaneously for every possible $q$. This is
exactly the inequality required by Definition~\ref{def:strong-degeneracy}.
For $0\le q\le D_g$, put
\[
C_g(q)=
\left\lceil\rho_g\left(\alpha_g-\frac q2\right)\right\rceil
\]
and
\begin{equation}\label{eq:T-g}
T(g)=\max_{0\le q\le D_g}
\left(q+1+\left\lfloor\log C_g(q)\right\rfloor\right).
\end{equation}

\begin{theorem}\label{thm:general-strong-girth}
For every integer $g\ge5$, every finite simple planar graph of girth at least
$g$ is strongly $T(g)$-degenerate.
\end{theorem}

The proof follows the explanation above in three steps. First, Euler's
formula produces a variable vertex of local charge less than $\alpha_g$.
Second, the definition of $\rho_g$ converts this charge inequality into
\eqref{eq:strong-reduction} with $t=T(g)$. Third, the reduction is iterated
and Lemma~\ref{lem:constraint-order} gives the required strong ordering.

\begin{proof}
Consider a plane constraint system of girth at least $g$ in which every
constrain vertex has degree at most $D_g$. We show that it contains a variable vertex
satisfying \eqref{eq:strong-reduction} with $t=T(g)$.

If a variable vertex has degree at most $1$ in the whole auxiliary graph, the
claim is immediate. Indeed, $T(g)\ge D_g+1$ by \eqref{eq:T-g}. If the
vertex is isolated, then $q_P(v)=W_P(v)=0$. If its sole incident edge leads
to a constrain vertex, then $q_P(v)=0$ and
$W_P(v)\le 2^{D_g-1}-1$; if it is a direct edge, then $q_P(v)=1$ and
$W_P(v)=0$.

We may therefore assume that every component of the auxiliary graph has
minimum degree at least $2$. Fix one component and restrict $P$ to that
component; finding a reducible variable vertex there is sufficient. Write
$n=|X|$, let $m_r$ be the number of constrain vertices of degree $r$, and let
$e_D$ be the number of direct edges in this component. Since each facial
boundary contains a cycle, the usual Euler estimate for a plane graph of
girth at least $g$ gives
\[
e_D+\sum_{r=2}^{D_g}r m_r
\le \alpha_g\left(n+\sum_{r=2}^{D_g}m_r-2\right).
\]
To turn this global inequality into a statement about one variable vertex,
we distribute each direct edge equally between its endpoints. An incidence
with a constrain vertex of degree $r$ receives the remaining local
coefficient $1-\alpha_g/r$. Accordingly, for $v\in X$, define
\[
L(v)=\frac{q_P(v)}2+
\sum_{a\in N_P(v)\cap A}
\left(1-\frac{\alpha_g}{d_P(a)}\right).
\]
Summing the preceding planar inequality over the variable vertices yields
\[
\sum_{v\in X}L(v)
=e_D+\sum_{r=2}^{D_g}(r-\alpha_g)m_r
\le\alpha_g n-2\alpha_g<\alpha_g n.
\]
Hence some variable vertex $v$ satisfies $L(v)<\alpha_g$. If
$q=q_P(v)$, then $q<2\alpha_g$, and therefore $q\le D_g$. By the
definition of $\rho_g$,
\[
\begin{aligned}
W_P(v)
&=\sum_{a\in N_P(v)\cap A}
  \bigl(2^{d_P(a)-1}-1\bigr)\\
&\le\rho_g\sum_{a\in N_P(v)\cap A}
  \left(1-\frac{\alpha_g}{d_P(a)}\right)\\
&<\rho_g\left(\alpha_g-\frac q2\right).
\end{aligned}
\]
Since $W_P(v)$ is an integer, this implies
\[
1+W_P(v)\le C_g(q)<2^{T(g)-q},
\]
where the last inequality follows from \eqref{eq:T-g}. Thus $v$ satisfies
\eqref{eq:strong-reduction}.

Starting with a plane graph $G$ and no constrain vertices, the degree bound
on constrain vertices is vacuous. At every reduction the new constrain
vertex has degree $q_P(v)\le D_g$, and
the girth is preserved. The reduction can therefore be repeated. Applying
Lemma~\ref{lem:constraint-order} completes the proof.
\end{proof}

The formula in Theorem~\ref{thm:general-strong-girth} is explicit and involves
only a finite computation. Its values are as follows.

\begin{table}[htbp]
\centering
\begin{tabular}{|c|c|c|c|c|}
\hline
$g$ & $\alpha_g$ & $D_g$ & $\rho_g$ & $T(g)$ \\
\hline
$5$ & $5/3$ & $3$ & $27/4$ & $5$ \\
\hline
$6$ & $3/2$ & $2$ & $4$ & $4$ \\
\hline
$7$ & $7/5$ & $2$ & $10/3$ & $4$ \\
\hline
$g\ge8$ & $\le4/3$ & $2$ & $\le3$ & $3$ \\
\hline
\end{tabular}
\caption{The strong-degeneracy computation for planar girth $g\ge5$.}
\label{tab:strong-computation}
\end{table}

For clarity, when $g=5$ the four values $C_5(q)$, for $q=0,1,2,3$, are
$12,8,5,2$. When $g=6$ they are $6,4,2$; when $g=7$ they are $5,3,2$;
and at $g=8$ they are $4,3,1$. Substitution in \eqref{eq:T-g} gives the
last column of Table~\ref{tab:strong-computation}. For $g\ge6$, we have
$D_g=2$ and
\[
\rho_g\left(\alpha_g-\frac q2\right)
=\frac{2\alpha_g-q}{2-\alpha_g},
\]
which is increasing in $\alpha_g$. Since $\alpha_g$ decreases with $g$,
each relevant quantity for $g\ge8$ is maximized at $g=8$. Hence $T(g)\le3$;
the reverse inequality $T(g)\ge D_g+1=3$ follows from \eqref{eq:T-g}, so
$T(g)=3$ throughout this range.

The general computation can now be converted directly into the advertised
inversion-diameter bounds by Theorem~\ref{thm:strong-diameter}.

\begin{proof}[Proof of Theorem~\ref{thm:girth5-intro}]
The computation in Table~\ref{tab:strong-computation} gives $T(5)=5$.
Theorems~\ref{thm:general-strong-girth} and
\ref{thm:strong-diameter} now give $\diam(I(G))\le5$.
\end{proof}

\begin{proof}[Proof of Theorem~\ref{thm:girth6-intro}]
The computation in Table~\ref{tab:strong-computation} shows that $T(6)=4$.
Thus $G$ is strongly $4$-degenerate by
Theorem~\ref{thm:general-strong-girth}, and
Theorem~\ref{thm:strong-diameter} gives $\diam(I(G))\le4$.
\end{proof}

The same argument gives inversion diameter at most $3$ for planar graphs of
girth at least $8$. At girth $7$, the uniform strong-degeneracy computation
gives $4$, while the sharper bound $3$ is known from Arana et
al.~\cite{ABBC+}.

\subsection{The limiting case: girth four}

We now explain why the computation in
Theorem~\ref{thm:general-strong-girth} does not extend to $g=4$. The uniform
proof uses $\alpha_g=g/(g-2)$ and converts the local Euler
charge of a constrain vertex of degree $r$ into its contribution to
$W_P(v)$ through the ratio
\[
\frac{2^{r-1}-1}{1-\alpha_g/r}.
\]
For $g\ge5$, we have $\alpha_g<2$, so the denominator is positive for every
$r\ge2$. At girth $4$, however, $\alpha_4=2$. A binary constrain vertex
therefore satisfies
\[
1-\frac{\alpha_4}{2}=0,
\qquad\text{whereas}\qquad
2^{2-1}-1=1.
\]
Thus a binary constrain vertex contributes nothing to the Euler charge but
still contributes one forbidden nonempty subset to $W_P(v)$. Equivalently,
the definition of $\rho_g$ contains a zero denominator when $g=4$, so no
finite value of $\rho_4$ can make the conversion used in the general proof.
This is the precise point where that computation fails.

To repair the failure, we need an independent estimate for binary constrain
vertices. Canonicality allows us to suppress them, apply a second planar
estimate, and combine it with the triangle-free Euler estimate. The
coefficients are chosen to dominate every possible contribution to
$W_P(v)$. The following lemma supplies the variable vertex at which the
reduction can continue.

\begin{lemma}\label{lem:girth4-reducible}
Let $P$ be a nonempty canonical triangle-free plane constraint system in
which every constrain vertex has degree at most $5$. Then some variable vertex $v$
satisfies
\[
q_P(v)\le4
\qquad\text{and}\qquad
W_P(v)<2^{7-q_P(v)}-1.
\]
\end{lemma}

We briefly describe the calculation before giving the proof. The first
planar inequality uses triangle-freeness and controls direct edges together
with constrain vertices of degrees $3$, $4$, and $5$. The second inequality
is obtained after suppressing binary constrain vertices. A weighted average
of these inequalities gives a negative total charge. If no variable vertex
satisfied the conclusion, however, the contribution of every variable vertex
would be nonnegative.

\begin{proof}
Write $n=|X|$ and let $m_r$ be the number of constrain vertices of degree
$r$. We use $M$ for the number of nonbinary constrain vertices and $e_I$
for the number of incidences involving them; thus
\[
M=m_3+m_4+m_5,\qquad
e_I=3m_3+4m_4+5m_5.
\]
Let $e_D$ be the number of direct edges and set $N=n+M$. Assume first that
$N\ge3$. Since $P$ is triangle-free and planar,
\[
e_D+2m_2+e_I\le2(n+m_2+M)-4,
\]
and hence
\begin{equation}\label{eq:g4-first-planar}
e_D+e_I\le2N-4.
\end{equation}

Suppress every binary constrain vertex. The resulting plane graph is simple:
canonicality excludes two binary constrain vertices on the same pair, while a
direct edge on that pair would form a triangle in $P$. Therefore
\begin{equation}\label{eq:g4-second-planar}
e_D+m_2+e_I\le3N-6.
\end{equation}

Taking $12/13$ of \eqref{eq:g4-first-planar} and $1/13$ of
\eqref{eq:g4-second-planar}, we obtain
\begin{equation}\label{eq:g4-combined-planar}
e_D+e_I+\frac{m_2}{13}
\le \frac{27}{13}N-\frac{54}{13}.
\end{equation}

We now distribute \eqref{eq:g4-combined-planar} among the variable vertices.
Each endpoint receives one half of a direct edge. Each incidence with a
binary constrain vertex receives coefficient $1/26$, and an incidence with
a constrain vertex of degree $r\ge3$ receives coefficient
$1-27/(13r)$. Finally, we subtract $27/13$ at every variable vertex. This
choice makes the sum of the local charges equal to the left side of
\eqref{eq:g4-combined-planar} minus $(27/13)N$.

For a variable vertex $v$, let $h_r(v)$ be the number of incident constrain
vertices of degree $r$, and define
\[
\gamma(v)=\frac{q_P(v)}2+\frac{h_2(v)}{26}
+\sum_{r=3}^5\left(1-\frac{27}{13r}\right)h_r(v)
-\frac{27}{13}.
\]
Summing over $X$ and using \eqref{eq:g4-combined-planar} gives
\begin{equation}\label{eq:g4-negative-charge}
\sum_{v\in X}\gamma(v)\le-\frac{54}{13}.
\end{equation}

Suppose that the conclusion fails at every variable vertex. If
$q_P(v)\ge5$, then $q_P(v)/2>27/13$, so $\gamma(v)\ge0$. Now suppose that
$q_P(v)\le4$. A constrain vertex of degree $2,3,4$, or $5$ contributes, respectively,
$1,3,7$, or $15$ to $W_P(v)$. Its positive coefficient in $\gamma(v)$ is at
least $1/26$ times that contribution, since
\[
\frac1{26}=\frac1{26},\qquad
\frac4{13}\ge\frac3{26},\qquad
\frac{25}{52}\ge\frac7{26},\qquad
\frac{38}{65}>\frac{15}{26}.
\]
The failure of the conclusion gives
$W_P(v)\ge2^{7-q_P(v)}-1$, and consequently
\[
\gamma(v)\ge
\frac{q_P(v)}2+\frac{2^{7-q_P(v)}-1}{26}-\frac{27}{13}.
\]
For $q_P(v)=0,1,2,3,4$, the first two terms on the right have numerators
\[
127,\quad76,\quad57,\quad54,\quad59
\]
over the common denominator $26$. Each is at least
$54/26=27/13$, and hence $\gamma(v)\ge0$. This contradicts
\eqref{eq:g4-negative-charge}.

It remains to consider $N<3$, which is trivial. The conclusion follows immediately.
\end{proof}

\begin{theorem}\label{thm:strong-girth4}
Every finite simple planar graph of girth at least $4$ is strongly
$7$-degenerate.
\end{theorem}

Lemma~\ref{lem:girth4-reducible} is designed so that its two inequalities are
exactly \eqref{eq:strong-reduction} with $t=7$. We may therefore iterate the
reduction and use Lemma~\ref{lem:constraint-order}.

\begin{proof}
Fix a plane embedding of $G$ and start with $G$ as a constraint system with
no constrain vertices. Lemma~\ref{lem:girth4-reducible} supplies a vertex satisfying
\eqref{eq:strong-reduction} with $t=7$. Its direct degree is at most $4$, so
the new constrain vertex created by the reduction has degree at most $4$. The reduced
system is again canonical and triangle-free, and the degrees of all constrain
vertices are at most $4$. We may therefore continue until no variable
vertex remains. The result now
follows from Lemma~\ref{lem:constraint-order}.
\end{proof}

The girth-$4$ argument is now complete. As in the general case, no additional
vector extension is needed: we apply the theorem of Havet et
al.~\cite{HHR24} directly to the strong ordering just constructed.

\begin{proof}[Proof of Theorem~\ref{thm:girth4-intro}]
By Theorem~\ref{thm:strong-girth4}, the graph $G$ is strongly
$7$-degenerate. Theorem~\ref{thm:strong-diameter} gives
$\diam(I(G))\le7$.
\end{proof}

\section*{Declarations}
The authors acknowledge the use of OpenAI’s ChatGPT during the preparation of the manuscript. It was used to improve the language, organization, and presentation of the manuscript and to
assist in revising preliminary drafts. In particular, the idea in Theorem \ref{thm:main} arose from an interaction with ChatGPT. After using this tool the authors reviewed and edited the content as needed and take full responsibility for the content of the publication. No mathematical statement, proof, constant, numerical value or reference in this paper was generated without author
verification. The authors take full responsibility for this paper.


\end{document}